\documentclass[a4paper,11pt,reqno]{amsart}

\usepackage[top=3cm, bottom=3cm, left=3cm, right=3cm]{geometry}

\usepackage[utf8]{inputenx}
\usepackage[english]{babel}
\usepackage{microtype}

\usepackage{mathtools}
\usepackage{amssymb}
\usepackage{amsfonts}
\usepackage{amsthm}
\usepackage{mathrsfs}
\usepackage{stmaryrd}

\usepackage{graphicx}
\usepackage{tikz}
\usepackage{tikz-cd}
\usepackage{dynkin-diagrams}
\usepackage[all]{xy}

\usepackage{tensor}
\usepackage{centernot}
\usepackage{faktor}
\usepackage{xfrac}
\usepackage{esint}
\usepackage{bbm}

\usepackage{hyperref}

\usepackage{xcolor}
\usepackage{array, tabularx}
\usepackage{mathdots}
\usepackage{comment}
\usepackage{paralist}
\usepackage{xcolor}

\theoremstyle{plain}
\newtheorem{thm}{Theorem}[section]
\newtheorem{prop}[thm]{Proposition}
\newtheorem{lem}[thm]{Lemma}

\theoremstyle{definition}
\newtheorem{defn}[thm]{Definition}
\newtheorem{ex}[thm]{Example}

\theoremstyle{remark}
\newtheorem{rem}[thm]{Remark}

\DeclareMathOperator{\Ric}{Ric}

\newcommand{\bigslant}[2]{{\raisebox{.2em}{$#1$}\left/\raisebox{-.2em}{$#2$}\right.}}

\begin{document}
	\newpage
	
	\title{Futaki invariant on lcK manifolds with potential}

  \author{Giacomo Perri}

	\address{Giacomo Perri \newline
		\textsc{\indent Institut for Matematik, Aarhus University\newline 
			\indent 8000, Aarhus C, Denmark}}
	\email{g.perri@math.au.dk}

\thanks{The author is supported by the Sapere Aude project $\lq\lq$Conformal geometry: metrics and cohomology" at Aarhus University}
	%\date{\today}

\begin{abstract}
    We prove that the Futaki invariant introduced by Futaki, Hattori, and Ornea vanishes identically on every compact locally conformally K\"ahler manifold with potential, thereby answering a question raised by Ornea and Verbitsky.
\end{abstract}
 \maketitle

%\tableofcontents

\section{Introduction}

The Futaki invariant was originally introduced as an obstruction to the
existence of K\"ahler--Einstein metrics on Fano manifolds and has since
played a central role in K\"ahler geometry \cite{Fut83}. Futaki, Hattori, and Ornea
\cite{FHO13} extended this invariant to arbitrary compact complex manifolds.
In this setting, the invariant
obstructs the existence of a volume form for which the top exterior power
of the associated Ricci form is proportional to the volume form itself.
They investigated this invariant from the point of view of locally conformally
K\"ahler (lcK) geometry. In particular, they proved that it vanishes on every
compact Vaisman manifold. On the other hand, they exhibited a one-point blow-up of
a Hopf surface for which the invariant is nonzero, showing that such a
vanishing result cannot hold for arbitrary lcK manifolds.

An important class in lcK geometry is that of lcK manifolds with potential.
These form a large and geometrically distinguished subclass of lcK
manifolds, which includes all Vaisman manifolds and all Hopf manifolds.
Motivated by the vanishing theorem for Vaisman manifolds \cite[Theorem~1.2]{FHO13}, Ornea and
Verbitsky asked whether the Futaki invariant vanishes on all compact
lcK manifolds with potential \cite{OV24b}. 
In this note, we answer this question in full.

\begin{thm}\label{main}
Let $(M,J)$ be a compact complex manifold admitting an lcK metric with
potential. Then its Futaki invariant vanishes identically.
\end{thm}

Theorem~\ref{main} extends both the vanishing theorem of Futaki, Hattori, and Ornea for compact Vaisman manifolds and the author's previous result for Hopf manifolds \cite{Per26}. The proof differs from the arguments used
in these two special cases and instead relies on a general vanishing
criterion that is independent of lcK geometry.

\begin{prop}\label{vanishing if existence}
    Let $(M,J)$ be a compact complex manifold admitting a nowhere-vanishing
holomorphic vector field. Then its Futaki invariant vanishes identically.
\end{prop}

Proposition~\ref{vanishing if existence} follows from a localization formula for the Futaki--Morita integral invariants. We then
apply this result to lcK manifolds with potential. In complex
dimension at least three, a suitable K\"ahler covering carries a
locally free holomorphic $\mathbb{C}^*$-action commuting with the
deck group. Its infinitesimal generator is nowhere-vanishing and
descends to the quotient. The case of complex surfaces is treated
separately using results from the classification of compact complex surfaces.

Section~\ref{Preliminaries} recalls the necessary background on lcK manifolds with potential
and the Futaki invariant, while Section~\ref{main section} is devoted to the proof of
Theorem~\ref{main}.

\hfill

\noindent {\it Acknowledgments.}
The author acknowledges the use of ChatGPT 5.6 Plus for discussions related to this work. The author thanks Alexandra Otiman for reading a preliminary version of this note and for helpful suggestions.

\section{Preliminaries}\label{Preliminaries}
In this section, we briefly recall the main definitions and properties of locally conformally K\"ahler manifolds and of the Futaki invariant. For further details, we refer to \cite{OV24b} and \cite{FHO13}, respectively.

\subsection{LcK manifolds}
    A \emph{locally conformally K\"ahler} (lcK) manifold is a connected Hermitian manifold $(M,J,g,\omega,\theta)$ of complex dimension $n\geq 2$, where
$\omega:=g(J\cdot,\cdot)$
is the fundamental $2$-form associated with the Riemannian metric $g$, and $\theta$ is a closed $1$-form, called the \emph{Lee form} of $\omega$, satisfying
\begin{align}\label{lcKdefinition}
d\omega=\theta\wedge\omega.
\end{align}
If $\theta=0$, then $\omega$ is Kähler. If $\theta=d\varphi$ is exact, then the conformally equivalent metric
$e^{-\varphi}\omega$
is K\"ahler. In this case, $\omega$ is called \emph{globally conformally K\"ahler} (gcK). An lcK manifold whose Lee form is parallel with respect to the Levi--Civita connection is called a \emph{Vaisman manifold}. Throughout the paper, all lcK structures are assumed to be strict, \textrm{i.e.}, their Lee forms are not exact. Compact strict lcK manifolds do not admit K\"ahler structures \cite[Corollary 2.2]{Vai80}.

Equivalently, an lcK manifold is a complex manifold $(M,J)$ admitting a
K\"ahler covering $(\tilde{M},\tilde{\omega})$ whose deck transformation
group $\Gamma$ acts by K\"ahler homotheties; that is, for every
$\gamma\in\Gamma$, there exists a positive constant $c_\gamma>0$ such
that
$\gamma^*\tilde{\omega}=c_\gamma\tilde{\omega}.$
A smooth function $f$ on $\tilde{M}$ is called \emph{automorphic} if for each $\gamma \in \Gamma$, $\gamma^* f = c_\gamma f$.
\begin{defn}
    Let $(M,J,\omega,\theta)$ be an lcK manifold, and let
$(\tilde{M},\tilde{\omega})$ be a K\"ahler covering of $M$. We say
that $\omega$ is an \emph{lcK metric with potential} if $\tilde{M}$ admits
a global positive automorphic K\"ahler potential
$$
\varphi\colon\tilde{M}\longrightarrow\mathbb{R}_{>0}
$$
such that
$$
dd^c\varphi=\tilde{\omega}.
$$
If the automorphic potential $\varphi$ is proper, meaning that the
preimage of every compact subset is compact, then we say that $\omega$
is an lcK metric with \emph{proper} potential; otherwise, we say that it is an lcK metric
with \emph{improper} potential. Similarly, we say that $M$ is an \emph{lcK manifold
with (proper) potential} if it admits an lcK metric with (proper)
potential.
\end{defn}

\begin{rem}\label{proper potential}
If a compact complex manifold $(M,J)$ admits an lcK metric with
potential, then it also admits a possibly different lcK metric with
proper potential \cite[Proposition 2.13]{OV16}.
\end{rem}

\begin{ex}
    A \emph{(primary) Hopf manifold} $H \equiv H_{\gamma}$ is a quotient 
\begin{equation}
    H := \bigslant{\mathbb{C}^n \setminus \{0\}}{<\gamma>},
\end{equation}
where 
$$\gamma: \mathbb{C}^n \to \mathbb{C}^n$$
is a biholomorphic contraction centered at zero. If $\gamma$ is linear, then $H_\gamma$ is called a \emph{linear Hopf manifold}.
Hopf manifolds are lcK manifolds with potential \cite{OV23}.

A \emph{secondary Hopf manifold} is a quotient of a Hopf manifold by a finite
group freely acting on it.
\end{ex}
    
In complex dimension at least three, compact lcK manifolds with
potential can be characterized as complex submanifolds of
linear Hopf manifolds.

\begin{thm}[\cite{OV10}]
Let $(M,J)$ be a compact complex manifold of complex dimension at least
three. Then $M$ is an lcK manifold with potential if and only if it
admits a holomorphic embedding into a linear Hopf manifold.
\end{thm}

\subsection{The Futaki invariant in complex geometry}
Let $(M,J)$ be a compact complex manifold of complex dimension $n$, equipped with a volume form $\Omega$. On a complex manifold, giving a volume form $\Omega$ is equivalent to giving a Hermitian metric $H_\Omega$ on the canonical bundle $K_M$. Indeed, if $\nu =dz_1\wedge \cdots \wedge dz_n$ is a local section of $K_M$, then the Hermitian metric $H_\Omega$ is given by $H_\Omega(\nu,\nu) := \frac{i^{n^2}\nu \wedge \overline{\nu}}{\Omega}$.

The Ricci form $\Ric_\Omega$ of the volume form $\Omega$ is defined by
$$\Ric_\Omega:=-i\,\Theta\big(K_M,H_\Omega),$$
where $\Theta\big(K_M,H_\Omega)$ denotes the Chern curvature of $K_M$ with respect to $H_\Omega$.
In particular, the Ricci form associated with the volume form of a Hermitian metric $\omega$ is the Chern--Ricci form of $\omega$, $\textsl{i.e.}$,
$$\Ric_{\omega^n}=\text{CRic}_\omega.$$
If locally $$\Omega \simeq_{\textrm{loc}} a \, idz_1\wedge d\overline{z}_1\wedge\cdots\wedge idz_n\wedge d\overline{z}_n,$$
then the associated Ricci form is given by 
$$\Ric_\Omega \simeq_{\textrm{loc}} -i\,\partial \overline{\partial} \log a.$$

Let $V \in H^0\big(M, T^{1,0}_M\big)$ be a holomorphic vector field on $M$. The divergence of $V$ with respect to the volume form $\Omega$ is defined by
$$\text{div}_\Omega V := \frac{\mathcal{L}_V \Omega}{\Omega} = \frac{d (\iota_V \Omega)}{\Omega} = \frac{\partial (\iota_V \Omega)}{\Omega},$$
where the last equality follows from the fact that $V$ is holomorphic. Here, $\mathcal{L}_V\Omega$ denotes the Lie derivative of $\Omega$ along
$V$, while $\iota_V$ denotes contraction with $V$. Locally, the divergence is given by 
$$\text{div}_\Omega V \simeq_{\textrm{loc}} V(\log a) + \sum_{l=1}^n \frac{\partial V_l}{\partial z_l},$$
where $V \simeq_{\textrm{loc}} \sum_{l=1}^n V_l \frac{\partial}{\partial z_l}$.

\begin{defn}[\cite{FHO13}] The Futaki invariant $\text{F}_M$ of $M$ is the linear map
   $$\text{F}_M:H^0\big(M, T^{1,0}_M\big) \to \mathbb{C}$$
   defined by
   \begin{equation}
       \text{F}_M(V) := \int_M \text{div}_\Omega V \, \text{Ric}_\Omega^n.
   \end{equation}
\end{defn}
As shown in \cite{FHO13}, this definition is independent of the choice
of the volume form $\Omega$.\\

Integrating by parts, we can rewrite $\mathrm{F}_M$ as
\begin{equation}\label{secondform}
    \mathrm{F}_M(V)
=
-\int_M
V\left(\frac{\text{Ric}_\Omega^n}{\Omega}\right)\Omega.
\end{equation}
Consequently, the Futaki invariant obstructs the existence of a volume form $\Omega$ satisfying
\begin{equation}
    \text{Ric}_\Omega^n=c\,\Omega
\end{equation}
for some constant $c$. Indeed, the existence of such a volume form
implies that $\mathrm{F}_M$ vanishes identically.\\

In the classical K\"ahler setting, the Futaki invariant can be realized as one of the Futaki--Morita integral invariants \cite{FM85}. Inspired by Bott's residue formula \cite{Bot67}, Futaki and Morita obtained localization formulas for these invariants; see also \cite{Che18} for a more recent treatment. As explained in \cite{FHO13}, the invariant $\mathrm{F}_M$ considered
here is itself a Futaki--Morita integral invariant. Consequently, the same
localization argument applies without any K\"ahler assumption. We only need the following consequence.

\begin{thm}\label{vanishing on nowhere zero hvf}
    Let $(M,J)$ be a compact complex manifold and let
$V\in H^0\bigl(M,T^{1,0}M\bigr)$
be a nowhere-vanishing holomorphic vector field. Then
$\mathrm{F}_M(V)=0.$
\end{thm}

\section{Vanishing of the Futaki invariant on lcK with potential}\label{main section}
This section proves our main result, Theorem~\ref{main}. We first derive Proposition~\ref{vanishing if existence} from Theorem~\ref{vanishing on nowhere zero hvf}. We then prove that, in complex dimension greater than two, every lcK manifold with potential admits a nowhere-vanishing holomorphic vector field, which yields the desired vanishing result in higher dimensions. We treat the surface case separately, using the classification of compact complex surfaces together with the vanishing of the Futaki invariant on Hopf surfaces (Lemma~\ref{Hopf surfaces vanishing}).

\begin{proof}[Proof of Proposition \ref{vanishing if existence}]
Let $W\in H^0\big(M,T^{1,0}M\big)$ be a nowhere-vanishing holomorphic
vector field, and let $V\in H^0\big(M,T^{1,0}M\big)$ be an arbitrary
nonzero holomorphic vector field.

Fix a Hermitian metric $h$ 
on $M$. Since $M$ is compact, $W$ is nowhere-vanishing, and $V$ is 
nonzero, the numbers 
$$ 
a:=\min_{x\in M}|W(x)|_h $$
and 
$$
b:=\max_{x\in M}|V(x)|_h 
$$ 
are well-defined positive real numbers. Choose $t\in\mathbb{C}^*$ such 
that 
$$ 
0<|t|<\frac{a}{b}. 
$$ 
Then $W+tV$ is a nowhere-vanishing holomorphic vector field on $M$. 
Indeed, for every $x\in M$, we have 
$$ 
\bigl|(W+tV)(x)\bigr|_h 
\geq |W(x)|_h-|t|\,|V(x)|_h 
\geq a-|t|b>0. 
$$ 
Finally, using the fact that the Futaki invariant vanishes on 
nowhere-vanishing holomorphic vector fields (Theorem~\ref{vanishing on nowhere zero hvf}) and is 
$\mathbb{C}$-linear, we obtain 
$$ 
0=\mathrm{F}_M(W+tV) 
 =\mathrm{F}_M(W)+t\mathrm{F}_M(V) 
 =t\mathrm{F}_M(V). 
$$ 
Since $t\neq0$, we conclude that $\mathrm{F}_M(V)=0$. As the choice of $V$ was arbitrary, we obtain $\mathrm{F}_M\equiv0$.
\end{proof}

\begin{rem}
As a consequence of the Poincar\'e--Hopf theorem \cite{Hop27}, the existence of a
nowhere-vanishing (holomorphic) vector field on a compact complex manifold
implies that its Euler characteristic $\chi$ is zero. However, there exist compact non-K\"ahler complex manifolds with vanishing Futaki invariant but admitting no such vector field. For example, consider the product
$M=\mathbb{CP}^n\times S$, where $n\geq 1$ and $S$ is an Inoue
surface of type $S_M$. This is a compact non-K\"ahler manifold
on which the Futaki invariant vanishes. Indeed, $S$ admits no
non-trivial holomorphic vector fields \cite[Proposition 2]{Ino74}, so every holomorphic
vector field on $M$ is induced by one on $\mathbb{CP}^n$.
By taking a product volume form, with the volume form on $\mathbb{CP}^n$ induced by the Fubini--Study K\"ahler--Einstein metric, we obtain the vanishing of the Futaki invariant on $M$ directly from \eqref{secondform}.
On the other hand, since
$\chi(\mathbb{CP}^n)=n+1\neq 0$, every holomorphic vector field
on $\mathbb{CP}^n$ has a zero. Consequently, $M$ admits no
nowhere-vanishing holomorphic vector field.
\end{rem}

We now establish the existence of a nowhere-vanishing holomorphic vector
field on compact lcK manifolds with potential in complex dimension at
least three.

\begin{prop}\label{holomorphic vf on lcK with pot}
Let $(M,J,\omega,\theta)$ be a compact lcK manifold with potential of
complex dimension $n\ge 3$. Then $M$ admits a nowhere-vanishing holomorphic
vector field.
\end{prop}
\begin{proof}
Remark~\ref{proper potential} ensures that $M$ carries an lcK metric
with proper potential. Fix such a metric, and let $\tilde{M}$ be
the corresponding K\"ahler $\mathbb{Z}$-cover of $M$. Let $\gamma$
be a generator of its deck transformation group, so that
$$
M\simeq\tilde{M}/\langle\gamma\rangle.
$$
By \cite[Theorem~7.9]{OV24a} and \cite[Theorem 7.10]{OV26}, $\tilde{M}$ carries a
locally free holomorphic $\mathbb{C}^*$-action
$$
\rho\colon\mathbb{C}^*\times\tilde{M}\longrightarrow\tilde{M}
$$
which commutes with $\gamma$, in the sense that
$$
\rho(\lambda,\gamma x)=\gamma\bigl(\rho(\lambda,x)\bigr)
$$
for every $(\lambda,x)\in\mathbb{C}^*\times\tilde{M}$. For
$x\in\tilde{M}$, let 
$$
(\mathbb{C}^*)_x
:=
\bigl\{\lambda\in\mathbb{C}^*
\mid \rho(\lambda,x)=x\bigr\}
$$
denote its stabilizer.
Since the action is locally free, $(\mathbb{C}^*)_x$ is discrete.

Let
$\tilde{W}\in H^0\bigl(\tilde{M},T^{1,0}\tilde{M}\bigr)$
be the holomorphic vector field generated by this action, defined by
$$
\tilde{W}_x
:=
\left.\frac{d}{dt}\right|_{t=0}\rho(e^t,x)
$$
for every $x\in \tilde{M}$. 

We claim that $\tilde{W}$ is nowhere-vanishing. Suppose, by
contradiction, that $\tilde{W}_x=0$ for some $x\in\tilde{M}$.
The flow of $\tilde{W}$ is given by
$$
\Phi_t(x)=\rho(e^t,x).
$$
Since $\tilde{W}_x=0$, the constant curve $c(t)=x$ is an
integral curve with $c(0)=x$. Therefore, by uniqueness of integral curves, we obtain
$$
\rho(e^t,x)=x
$$
for every $t\in\mathbb{C}$. Since the exponential map
$\exp\colon\mathbb{C}\to\mathbb{C}^*$ is surjective, this implies
$(\mathbb{C}^*)_x=\mathbb{C}^*$, which contradicts the local freeness of
the action.

It remains to show that $\tilde{W}$ descends to $M$. Using
$\rho(1,x)=x$ and differentiating the identity
$$\rho(e^t,\gamma x)=\gamma\big(\rho(e^t,x)\big)$$
at $t=0$,
we obtain
\begin{align*}
    \tilde{W}_{\gamma x}&=\left.\frac{d}{dt}\right|_{t=0}
  \rho(e^t,\gamma x)\\
  &=\left.\frac{d}{dt}\right|_{t=0}
  \gamma\bigl(\rho(e^t,x)\bigr)\\
  &=(d\gamma)_x\bigl(\tilde{W}_x\bigr).
\end{align*}
Thus, $\tilde{W}$ is $\gamma$-invariant and hence it descends to a holomorphic vector
field
$$
W\in H^0\bigl(M,T^{1,0}M\bigr).
$$
Since $\tilde{W}$ is nowhere-vanishing, $W$ is nowhere-vanishing as well.
\end{proof}

\begin{rem}\label{vanishing on Vaisman}
The vanishing of the Futaki invariant on compact Vaisman manifolds was
proved in \cite[Theorem 1.2]{FHO13}. It also follows immediately from
Proposition~\ref{vanishing if existence}. Indeed, the Lee vector field is a nowhere-vanishing holomorphic vector field on Vaisman manifolds \cite[Proposition 7.24]{OV24b}.
\end{rem}

The vanishing of the Futaki invariant on primary Hopf surfaces was already
established in \cite{Per26}. We give here a different proof based on
Proposition~\ref{vanishing if existence}, and show that the vanishing extends naturally to secondary Hopf surfaces.

\begin{lem}\label{Hopf surfaces vanishing}
The Futaki invariant vanishes on every Hopf surface, both primary and
secondary.
\end{lem}

\begin{proof}
Let $H_\gamma$ be a primary Hopf surface. Following \cite[p. 181]{Arn88} and
\cite[Appendix~A]{Hae85}, after a
biholomorphic change of coordinates, we may assume that $\gamma$ is in
the normal form
$$\gamma(z,w)=(\lambda z+a w^m,\mu w),$$
where
$$a,\lambda,\mu\in\mathbb{C}, \qquad 0<|\lambda|\leq |\mu|<1,\qquad
a(\lambda-\mu^m)=0,$$
and $m\in\mathbb N_{>0}$.

If $a=0$, consider the holomorphic vector field
$$V=z\frac{\partial}{\partial z}
+w\frac{\partial}{\partial w}$$
on $\mathbb C^2\setminus\{0\}$. It is nowhere-vanishing and satisfies
$$d\gamma(V)=V\circ\gamma,$$
hence, it descends to a nowhere-vanishing holomorphic vector field on
$H_\gamma$.

If $a\neq0$, the same argument applies to
$$V=mz\frac{\partial}{\partial z}
+w\frac{\partial}{\partial w},$$
which is again nowhere-vanishing and $\gamma$-invariant. Therefore,
in both cases, Proposition~\ref{vanishing if existence} yields $\mathrm{F}_{H_\gamma}\equiv0$.

Now let $H$ be a secondary Hopf surface. Then there exists a finite covering
$$
\pi\colon\hat H\longrightarrow H
$$
such that $\hat{H}$ is a primary Hopf surface. Let $X$ be a holomorphic vector field on $H$, and let $\Omega$ be a volume form on $H$. Since $\pi$ is a local biholomorphism, $X$ admits a unique holomorphic lift $\hat X$ to $\hat H$ such that
$d\pi(\hat X)=X\circ\pi$. Moreover, $\hat\Omega:=\pi^*\Omega$ is a volume form on $\hat H$.
We also note that the divergence and the Ricci form are compatible with the pullback $\pi^*$, namely,
$$\pi^* \text{div}_\Omega X=\text{div}_{\hat\Omega}\hat X, \qquad \pi^*\text{Ric}_\Omega=\text{Ric}_{\hat\Omega}. $$
Therefore, using the vanishing already proved for primary Hopf surfaces, we obtain
\begin{align*}
0=\mathrm{F}_{\hat{H}}(\hat{X})
&=
\int_{\hat{H}}
\text{div}_{\hat{\Omega}}\hat{X}\,
\text{Ric}_{\hat{\Omega}}^2 \\
&=
\int_{\hat{H}}
\pi^*\left(
\text{div}_{\Omega}X\,
\text{Ric}_{\Omega}^2
\right) \\
&=
\deg(\pi)
\int_H
\text{div}_{\Omega}X\,
\text{Ric}_{\Omega}^2 \\
&=
\deg(\pi)\,\mathrm{F}_H(X).
\end{align*}
Therefore, $\mathrm{F}_H(X)=0$ for every holomorphic vector field $X$ on $H$,
and hence $\mathrm{F}_H\equiv 0$.
\end{proof}

We now prove our main result.
\begin{proof}[Proof of Theorem \ref{main}]
Combining Proposition~\ref{vanishing if existence} and
Proposition~\ref{holomorphic vf on lcK with pot}, we obtain the vanishing
of the Futaki invariant on compact lcK manifolds with potential of complex
dimension at least three.

It remains to consider the case in which $M$ is a compact complex
surface. We first observe that $M$ cannot contain a rational curve. Indeed, suppose that there exists a non-constant holomorphic map
$$
f\colon\mathbb{P}^1\longrightarrow M,
$$
and let $\pi\colon(\tilde M,\tilde\omega)\to M$ be a K\"ahler covering associated with the given
lcK structure with potential, so that
$$
\tilde\omega=dd^c\varphi,
$$
for some smooth function $\varphi$.
Since $\mathbb{P}^1$ is simply connected, $f$ admits a lift
$\tilde f\colon\mathbb{P}^1\to\tilde M$, which is non-constant.
Since $\tilde\omega$ is K\"ahler, the form
$\tilde{f}^*\widetilde{\omega}$ is semipositive and not identically
zero. Therefore,
$$0<
\int_{\mathbb{P}^1}\tilde f^{*}\tilde\omega
=
\int_{\mathbb{P}^1}dd^c(\varphi\circ\tilde f)
=0,
$$
where the last equality follows from Stokes' theorem, a contradiction.
Thus, $M$ contains no rational curves and is therefore
minimal.

Assume first that $M$ is not of class $\mathrm{VII}$. Since $M$ is a minimal non-K\"ahler surface, it admits a Vaisman structure
by \cite[Theorem~25.6]{OV24b}. Hence, $\mathrm{F}_M\equiv 0$ by
Remark~\ref{vanishing on Vaisman}.

Suppose next that $M$ is of class $\mathrm{VII}$ with $b_2(M)>0$. By
\cite[Th\'eor\`eme~0.1]{DOT01}, if $M$ admits a nonzero holomorphic vector
field, then $M$ is a Kato surface. Since Kato surfaces do not admit lcK
metrics with potential \cite[Proposition 2.5]{IOP17}, $M$ admits no nonzero holomorphic vector fields. Consequently, $H^0\big(M,T^{1,0}M\big)=\{0\}$, and therefore $\mathrm{F}_M\equiv 0$.

Finally, suppose that $M$ is of class
$\mathrm{VII}$ and $b_2(M)=0$. Since $M$ is minimal, Bogomolov's classification theorem
\cite{Bog76,Tel94} implies that $M$ is either an Inoue surface or a Hopf surface.
The Inoue case is excluded because Inoue surfaces do not admit
lcK metrics with potential \cite[Corollary~4.13]{Oti18}. Therefore,
$M$ is a Hopf surface, and Lemma~\ref{Hopf surfaces vanishing} yields $\mathrm{F}_M\equiv 0$.
\end{proof}

\begin{rem}
We note that, beyond the lcK setting, there are other important classes
of compact non-K\"ahler manifolds admitting a nowhere-vanishing
holomorphic vector field and hence having vanishing Futaki invariant,
such as the Calabi--Eckmann manifolds.
A Calabi--Eckmann manifold $M_{p,q}$ (see, for example,
\cite[Example~2.5]{Ver05}) is constructed as follows. Fix integers
$p,q\geq 1$ and a complex number
$\tau\in\mathbb{C}\setminus\mathbb{R}$. Then $M_{p,q}$ is defined as
the quotient of
$
\big(\mathbb{C}^{p+1}\setminus\{0\}\big)
\times
\big(\mathbb{C}^{q+1}\setminus\{0\}\big)
$
by the $\mathbb{C}$-action
$$
t\cdot(z,w):=(e^t z,e^{\tau t}w).
$$
The nowhere-vanishing holomorphic vector field
$$
Z=\sum_{i=0}^{p}z_i\frac{\partial}{\partial z_i}
$$
is invariant under this action and descends to a nowhere-vanishing holomorphic
vector field on $M_{p,q}$. 
These manifolds do not admit lcK metrics
\cite[Remark~37.32]{OV24b}. They are also non-K\"ahler, since their
second Betti number vanishes.
\end{rem}

\bibliographystyle{alpha}
%\bibliography{Bibliography}

\end{document}